\documentclass{IET}

\usepackage{amsmath, amsthm, amscd, amsfonts, amssymb, graphicx, color, mathrsfs}
\newtheorem{theorem}{Theorem}

\newtheorem{corollary}{Corollary}
\newtheorem{definition}{Definition}
\newtheorem{lemma}{Lemma}
\newtheorem{rem}{Remark}

\begin{document}

\title{The index of Toeplitz operators on compact Lie groups and  simply connected closed 3-manifolds.}

\author{Duv\'an Cardona}
\affil{Pontificia Universidad Javeriana, Mathematics Department, Bogot\'a-Colombia}
\affil[1]{cardonaduvan@javeriana.edu.co}

\abstract{In this paper we use the notion of  operator-valued symbol in order to compute the index of Toeplitz operators on compact Lie groups. Our approach combines the Connes index theorem and the infinite-dimensional operator-valued symbolic calculus of Ruzhansky-Turunen.  We also give applications to the index of Toeplitz operators on simply connected closed $3$-manifolds  $\mathbb{M}\simeq \mathbb{S}^3\simeq \textnormal{SU}(2) ,$ by using, as a fundamental tool,  the   Poincar\'e theorem (see Perelman \cite{Perelman1,Perelman2,Perelman3,Perelman4}).  MSC 2010: Primary 58J20; Secondary  58J22, 43A77, 57M27.}

\maketitle

\tableofcontents

\section{Introduction}
From the interplay of the Connes index theorem for Fredholm modules and the  operator-valued symbolic calculus  of Ruzhansky and Turunen, in this paper we compute the index of Toeplitz operators acting on functions in compact Lie groups. Although, the point of departure of the index theory is the Atiyah-Singer index theorem, proved in 1963 in \cite{AS} (see also, the historical references \cite{atiyabott1,atiyabott2,AS1,AS2,AS3,AS44,AS4,AS5}) the analysis for  the index of Toeplitz operators started with the classical formula of Noether-Gohberg-Krein (see \cite{NF}) 
\begin{equation}\label{ToeplitzNGK}
\textnormal{ind}(PM_{f}P)=-\textnormal{wn}(f):=-\frac{1}{2\pi i}\int_{\mathbb{S}^1}f^{-1}{df}.
\end{equation}
In the index formula \eqref{ToeplitzNGK}, the function $f\in C^\infty(G)$ is invertible everywhere, $M_f$ is the multiplication operator by $f,$ and $P$ is the projection from $L^2(\mathbb{S}^1)$ into the Hardy space $H^1(\mathbb{S}^1),$ consisting of those functions in $L^2$ with negative Fourier coefficients vanishing. The main feature in \eqref{ToeplitzNGK} is that the left hand side is of analytical nature, but the right hand side has topological information given by minus the winding number of $f$ around of zero. This result was extended by V. Venugopalkrishna \cite{Venugopalkrishna} to the unit ball in $\mathbb{C}^n$ and by L. Boutet de Monvel to arbitrary strictly pseudoconvex domains \cite{LvM}. A similar formula for boundaries of strictly pseudo-convex domains
in $\mathbb{C}^n$ was announced by Dynin \cite{Dy}. For the spectral properties of Toeplitz operators and its index theory in several complex variables, we refer the reader to the references  Douglas \cite{Dou}, Guillemin \cite{Gui}, Boutet de Monvel and Guillemin \cite{LvM2}, Boutet de Monvel \cite{LvM3,LvM4}, Murphy \cite{Murphy1,Murphy2,Murphy3} as well as the monograph Upmeier\cite{Up}. The index theory in the operator-valued context can be found in Cardona \cite{CardonaIndex}. We refer the reader to Hong \cite{Hong} for a Lie-algebraic approach to the local index theorem on compact Lie groups and general compact homogeneous spaces.

In this paper we want to compute the index of Toeplitz operators on compact Lie groups by using the recent notion of operator valued symbol  \cite{Ruz,Ruz-T} in the sense of Ruzhansky and Turunen. Our main theorem can be announced as follows. Here $G$ is a compact Lie group, $\widehat{G}$ is its unitary dual, $e_G$ is the identity element of $G,$ and for every irreducible representation $[\xi]\in \widehat{G},$ $d_\xi:=\dim[\xi:G\rightarrow \textnormal{hom}(\mathbb{C}^{d_\xi})]$ denotes  the dimension of the representation space. 

\begin{theorem}[Index of Toeplitz operators]\label{maintheorem}
Let us consider a smooth   and invertible function $f$ on $G.$  Let  $\Pi:L^2(G)\rightarrow \Pi'$ be the orthogonal projection of a closed subspace $\Pi'\subset L^2(G)$. If  $T_f=\Pi M_{f} \Pi$ is the Toeplitz operator  with symbol $f,$ then $T_f$ extends to a  Fredhlom operator on $L^2(G)$ and its analytical index is given by
\begin{align}
\textnormal{ind}(T_f)=\int\limits_{G}\sum_{[\xi]\in\hat{G}}d_{\xi}\textnormal{Tr}[\sigma_{I_{\Pi,f}}(x)\xi(e_G)     ]dx,\,\,I_{\Pi,f}:=f^{-1}\underbrace{[{\Pi},f][{\Pi},{{f^{-1}}}]\cdots [{\Pi},f]}_{n+2-times} 
\end{align} if $n$ is odd,  or
\begin{align}
\textnormal{ind}(T_f)=\int\limits_{G}\sum_{[\xi]\in\hat{G}}d_{\xi}\textnormal{Tr}[\sigma_{ I_{\Pi,f} }(x)\xi(e_G)    ]dx,\,\,I_{\Pi,f}:=f^{-1}\underbrace{[{\Pi},f][{\Pi},f^{-1}]\cdots [{\Pi},f]}_{n+1-times},
\end{align}  if $n$ is even, where
\begin{itemize}
\item $\sigma_{I_{\Pi,f}}$ is the Ruzhansky-Turunen operator valued symbol associated to $I_{\Pi,f},$ and
\item $[A,B]$ is the commutator operator defined by the operators $A$ and $B.$
\end{itemize}
\end{theorem}
As a consequence of the previous theorem, we will prove that the index of Toeplitz operators can be written in terms of the algebraic information encoded in the unitary dual of a compact Lie group, and the operator-valued Fourier analysis that allow us to define the Ruzhansky-Turunen operator valued symbol for continuous operators on $C^\infty(G).$ 

It is important to mention that from Theorem \ref{maintheorem}, we can derive index formulae for Toeplitz operators on  3-dimensional closed manifolds. In fact, if   $\mathbb{M}$ is a  simply connected closed 3-manifold, the Poincar\'e conjecture/theorem proved by Perelman provides a diffeomorphism  $ \mathbb{M}\simeq \mathbb{S}^3\simeq \textnormal{SU}(2),$ (see Ruzhansky and Turunen \cite{Ruz}, pag. 578, and Perelman \cite{Perelman1,Perelman2,Perelman3,Perelman4}) that allow us to construct a global (operator-valued) pseudo-differential calculus on $\mathbb{M}.$ This  implies that the analysis employed for Toeplitz operators on $\textnormal{SU}(2)$ gives a similar construction for Toeplitz operators on $\mathbb{M}.$ So, if  $T_\omega=\Pi M_{\omega} \Pi$ is a Toeplitz operator  with symbol $\omega,$ such that $\omega^{-1}\in C^\infty(\mathbb{M}),$  we will prove  the following algebraic index formula (see Corollary \ref{maintheoremDC2}),
\begin{align}
\textnormal{ind}(T_\omega)=\int\limits_{\mathbb{M}}\sum_{\ell \in \frac{1}{2}\mathbb{N}_0}(2\ell+1)\textnormal{Tr}[\sigma_{ \omega^{-1} [\Pi,\omega][\Pi,\omega^{-1}][\Pi,\omega][\Pi,\omega^{-1}][\Pi,\omega]} (x)   \xi_\ell(e_\mathbb{M}))     ]dx, 
\end{align} where $e_\mathbb{M}=\Phi(e_\textnormal{SU(2)}).$

This paper is organized as follows. In Section \ref{Preliminaries} we present some basics on the matrix valued and operator valued quantizations procedure for (global) pseudo-differential operators on compact Lie groups. In Section \ref{indexsection} we study the index of Toeplitz operators on compact Lie groups.  Finally, in section \ref{indexsectioncompact} we investigate the index of Toeplitz operators on simply connected closed 3-manifolds.

\section{Global operators on compact Lie groups. Preliminaries}\label{Preliminaries}

In this section we consider the pseudo-differential calculus of Ruzhansky and Turunen on compact Lie groups. Here,  a Lie group is a group $G$ that at the same time is a finitedimensional
manifold of differentiability class $C^2$, in such a way that the two group operations of $G$, $x\mapsto x^{-1},$ and $(x,y)\mapsto x\cdot y$ are $C^2$-mappings, (see Duistermaat and Kolk \cite{Liegroups}).  Theorem 1.6.1 of \cite{Liegroups} shows that every
$C^2$ Lie group $G,$ in the sense of the previous remark, can be provided with the
structure of a real-analytic manifold for which it becomes a real-analytic Lie

\subsection{Operator-valued quantization of global operators on compact Lie groups}
In this subsection we present the Ruzhansky-Turunen operator valued quantization procedure for global operators on compact Lie groups. Throughout of this paper $G$ is a compact Lie group endowed with its normalised Haar measure $dg$. Our main tool is the Fourier analysis carried by a global Fourier transform. It can be defined as follows.
\begin{definition}[Operator-valued Fourier transform]
For $f\in \mathscr{D}'(G),$ the respective right-convolution operator $r(f):C^\infty(G)\rightarrow C^\infty(G)$ is defined by
\begin{eqnarray}
r(f)g=g\ast f,\,\,\,g\in C^{\infty}(G).
\end{eqnarray}
If $f\in L^2(G),$ the (right) global Fourier transform is defined by \begin{equation}
\widehat{f}\equiv r(f)=\int_{G}f(y)\pi_R(y)^*dy,
\end{equation}
where $\pi_R$ is the right regular representation on $G,$ defined by $\pi_R(x)f(y)=f(yx)$ and $\pi_{R}(x)^*=\pi_{R}(x^{-1}),$  $x\in G.$
\end{definition}

In terms of the Fourier transform, the Fourier inversion formula can be announced as follows.

\begin{theorem}[Fourier inversion formula]
Let us assume that $f\in C^{\infty}(G).$ Then $r(f)\pi_R(x)$ is a trace class operator on $L^2(G),$ and the indentity
\begin{equation}
f(x)=\textnormal{Tr}(r(f)\pi_R(x)),\,\,f\in C^{\infty}(G),
\end{equation}  holds true for every $x\in G.$
\end{theorem}

\begin{definition}[Ruzhansky-Turunen operator valued quantization]
If $\rho:G\rightarrow \mathscr{B}(C^\infty(G))$ is a continuous operator, the pseudo-differential operator $A$ associated to $\rho,$ is defined by
\begin{equation}
Af(x)=\textnormal{Tr}(\rho(x)r(f)\pi_R(x)),\,\,f\in C^{\infty}(G).
\end{equation}

\end{definition}
Conversely we have the following theorem due to Ruzhansky and Turunen.

\begin{theorem}
 If $A:C^\infty(G)\rightarrow C^\infty(G)$ is a continuous linear operator, then there exists an unique $\sigma_{A}:G \rightarrow \mathscr{B}(C^\infty(G))$ (called the operator-valued symbol of $A$) satisfying
\begin{equation}\label{RTQo}
Af(x)=\textnormal{Tr}(\sigma_A(x)r(f)\pi_R(x)),\,\,f\in C^{\infty}(G).
\end{equation} 
\end{theorem} \begin{proof}
The symbol $\sigma_A$ is defined as follows. Let $K_A\in C^\infty(G)\widehat\otimes\mathscr{D}'(G)$ be the distributional Schwartz kernel of $A$ and $R_{A}(x,y)=K(x,y^{-1}x)$ is the right-convolution kernel associated to $A.$ If  $x\in G,$ and $R_A(x)\in \mathscr{D}'(G)$ is defined by $(R_A(x))(y)=R_A(x,y)$ for every $y\in G,$ the (right) operator-valued symbol $\rho=\sigma_A$ associated to $A$ is defined by 
$
\sigma_A(x):=r(R_A(x)),\,\,\,x\in G.
$ It can be proved that this operator valued operator satisfies \eqref{RTQo} (see Ruzhansky and Turunen \cite{Ruz}, pag. 583)
\end{proof}

\subsection{Matrix-valued quantization of global operators on compact Lie groups} In this subsection we will present the matrix-valued quantization of global operators on compact Lie groups. There are two notions of continuous operators on smooth functions on compact Lie groups we can use. Namely, the one used in the case of general manifolds (based on the idea of {\em local symbols} as in  H\"ormander \cite{Hor2}) and, in a much more recent context, the one of global  operators on compact Lie groups as defined by M. Ruzhansky and V. Turunen \cite{RT1}\cite{RT2} (from {\em full symbols}, for which the notations and terminologies are taken from   \cite{Ruz-T}).  \\  
\\ Let us consider for every  compact Lie group $G$ its unitary dual $\widehat{G},$ that is the set of continuous, irreducible, and  unitary representations on $G.$  
As is the operator-valued quantization, the main tool in the matrix-valued quantization is a suitable notion of Fourier transform. We define it as follows.

\begin{definition}[Matrix-valued Fourier transform] Let us assume that $\varphi\in C^\infty(G).$ Then, the matrix-valued Fourier transform of $\varphi$ at $[\xi],$ is defined by
     $$  \mathscr{F}_G{\varphi}(\xi):=\int_{G}\varphi(x)\xi(x)^*dx.$$
\end{definition}
The Peter-Weyl theorem on compact Lie groups implies the following inversion formula.
\begin{theorem}[Fourier inversion formula]
Let us assume that $f\in L^1(G).$ Then, we have
$$\varphi(x)=\sum_{[\xi]\in \widehat{G}}d_{\xi}\text{Tr}(\xi(x)\mathscr{F}_G{\varphi}(\xi)), $$  for all $x\in G.$ In this case, the Plancherel identity on $L^2(G)$ is given by,
$$ \Vert \varphi \Vert^2_{L^2(G)}= \sum_{[\xi]\in \widehat{G}}d_{\xi}\text{Tr}(\widehat{\varphi}(\xi)\widehat{\varphi}(\xi)^*) =\Vert  \widehat{\varphi}\Vert^2_{ L^2(\widehat{G} ) } .$$
\end{theorem}

\noindent Notice that, since $\Vert A \Vert_{\textnormal{HS}}=\sqrt{\text{Tr}(AA^*)}$, the term within the sum is the Hilbert-Schmidt norm of the matrix $\widehat{\varphi}(\xi)$. The matrix-valued quantization procedure of Ruzhansky-Turunen can be introduced as follows.  Any linear operator $A$ on $G$ mapping $C^{\infty}(G)$ into $\mathcal{D}'(G)$ gives rise to a {\em matrix-valued symbol} $\sigma_{A}(x,\xi)\in \mathbb{C}^{d_\xi \times d_\xi}$ given by
\begin{equation}
\sigma_A(x,\xi)\equiv \xi(x)^{*}(A\xi)(x):=\xi(x)^{*}[A\xi_{ij}(x)]_{i,j=1,\cdots,d_\xi},
\end{equation}
which can be understood from the distributional viewpoint. Then it can be shown that the operator $A$ can be expressed in terms of such a symbol as \cite{Ruz-T}
\begin{equation}\label{mul}Af(x)=\sum_{[\xi]\in \widehat{G}}d_{\xi}\text{Tr}[\xi(x)\sigma_A(x,\xi)\widehat{f}(\xi)]. 
\end{equation} We will denote by $\sigma_A(\cdot)$ and $\sigma_A(\cdot,\cdot)$ to the operator-valued symbol and the matrix-valued symbol associated to $A$ respectively.
\begin{lemma}[Matricial symbols vs operator-valued symbols] 
Theorem 10.11.16 in Ruzhansky and Turunen \cite{Ruz} gives the identity $\sigma_A(x,\xi)=\sigma_{\sigma_A(x)}(y,\xi):=\xi(y)^*(\sigma_{A}(x)\xi)(y),$ for all $y\in G.$ In particular, if $y=e_{G}$ is the identity element in $G,$
\begin{eqnarray}\label{connection}
\sigma_A(x)\xi(e_G)=\sigma_A(x,\xi),\,x\in G,\,\,[\xi]\in \widehat{G}.
\end{eqnarray} 
\end{lemma}

Now, we want to introduce Sobolev spaces and, for this, we give some basic tools. \noindent Let $\xi\in Rep(G):=\cup \widehat{G},$ if $x\in G$ is fixed, $\xi(x):\mathbb{C}^{d_\xi}\rightarrow \mathbb{C}^{d_\xi}$ is an unitary operator and $d_{\xi}:=\dim \mathbb{C}^{d_\xi} <\infty.$ There exists a non-negative real number $\lambda_{[\xi]}$ depending only on the equivalence class $[\xi]\in \hat{G},$ but not on the representation $\xi,$ such that $-\mathcal{L}_{G}\xi(x)=\lambda_{[\xi]}\xi(x);$ here $\mathcal{L}_{G}$ is the Laplacian on the group $G$ (in this case, defined as the Casimir element on $G$). Let  $\langle \xi\rangle$ denote the function $\langle \xi \rangle=(1+\lambda_{[\xi]})^{\frac{1}{2}}$.  
\begin{definition}\label{sov} For every $s\in\mathbb{R},$ the {\em Sobolev space} $H^s(G)$ on the Lie group $G$ is  defined by the condition: $f\in H^s(G)$ if only if $\langle \xi \rangle^s\widehat{f}\in L^{2}(\widehat{G})$. 
\end{definition}
The Sobolev space $H^{s}(G)$ is a Hilbert space endowed with the inner product $$\langle f,g\rangle_{H^s(G)}=\langle J^{s}f, J^{s}g\rangle_{L^{2}(G)}$$ where, for every $r\in\mathbb{R}$, $J^{s}:H^r(G)\rightarrow H^{r-s}(G)$ is the bounded pseudo-differential operator (Bessel potential) with symbol $\sigma_{J^s}(x,\xi):=\langle \xi\rangle^{s}I_{\xi}$.
\begin{rem}\label{lem} In this paper the notion of Sobolev spaces $H^{s}(G)$ is essential, we will use this spaces in the proof of Lemma \ref{tripleta} and for description of global operators. An important fact is that   every global operator $T$ of order $m$ is a bounded operator from $H^{s}(G)$ into $H^{s-m}(G)$ (see Ruzhansky and Turunen \cite{RT1}).
\end{rem}
Now we introduce, for every $m\in\mathbb{R}$, the H\"ormander class $\Psi^{m}(G)$ of pseudo-differential operators of order $m$ on the compact Lie group $G$. As a compact manifold we consider $\Psi^{m}(G)$ as the set of those operators which, in all local coordinate charts, give rise to pseudo-differential operators in the H\"ormander class $\Psi^{m}(U)$ for an open set $U \subset \mathbb{R}^n$, characterized by symbols satisfying the usual estimates \cite{Hor2}
\begin{equation}
|\partial_{x}^{\alpha}\partial_{\xi}^{\beta}\sigma(x,\xi)|\leq C_{\alpha,\beta}\langle \xi\rangle^{m-|\beta|},
\end{equation}
for all $(x,\xi) \in T^*U \cong \mathbb{R}^{2n}$ and $\alpha,\beta\in \mathbb{N}^n$. This class contains, in particular, differential operator of degree $m>0$ and other well-known operators in global analysis such as heat kernel operators.  
The class  The H\"ormander classes $\Psi^{m}(G)$ where characterized in \cite{Ruz,RWT} by the condition: $A\in \Psi^{m}(G)$ if only if its matrix-valued symbol $\sigma_{A}(x,\xi)$ satisfies the inequalities
\begin{equation}\label{cero}
\Vert \partial_{x}^{\alpha}\mathbb{D}^{\beta}\sigma_{A}(x,\xi)\Vert_{op} \leq C_{\alpha,\beta} \langle \xi\rangle^{m-|\beta|},
\end{equation}
for every $\alpha,\beta\in \mathbb{N}^n.$ For a rather comprehensive treatment of this global calculus we refer to \cite{Ruz}.

\subsection{Fredholm operators on Hilbert spaces, Fredholm modules on associative algebras and the Connes index theorem}
The index is defined for a broad class of operators called Fredholm operators. Now,  we introduce this notion in more detail. For $X,Y$ normed spaces $B(X,Y)$ is the set of bounded linear operators from $X$ into $Y.$ In particular, if $X=Y=H,$ $B(H)\equiv B(H,H)$ denotes the algebra of bounded operators on $H.$ Here, we consider $H=L^2(G).$
\begin{definition}
 If $H_1$ and  $H_2$ are Hilbert spaces, the closed and densely defined operator $A:H_1\rightarrow H_2$ is Fredholm if only if $\text{Ker}(A)$ is finite dimensional and $A(H_1)=\text{Rank}(A)$ is a closed subspace of $H_2$ with finite codimension. In this case, the index of $A$ is defined by $\text{Ind}(A)=\dim  \text{Ker}(A)-\dim  \text{Coker}(A).$ The index formula also can be written as
 $$\text{ind}(A)=\text{dim}\text{Ker} (A)-\text{dim}\text{Ker} (A^*).  $$
\end{definition}

In our analysis we use the Connes index theorem for  odd Fredholm modules on associative algebras. We recall this definition as follows.

\begin{definition}
Let $A$ be an associative algebra over $\mathbb{C}$. An odd Fredholm module over $A$ is a triple $(A,\pi,F)$ consisting of:
\begin{itemize}
\item a Hilbert space $H,$
\item a representation $\pi$ of $A$ as bounded operators on $H,$
\item a self-adjoint operator $F$ such that $F^2=I$ and $[F,\pi(a)]$ is a compact operator on $H$ for all $a\in A.$
\end{itemize}
\end{definition} Additionally, if there exist $p$ such that $[F,\pi(a)]\in L^{p}(H)=\{T\in B(H,H):\sum_\nu [s_\nu(T)]^p<\infty\},$ (here, $\{s_\nu(T) \}_\nu $ denotes the sequence of singlular values of $T$)  we say that the module $(A,\pi,F)$ is $p$-summable.
The corresponding Connes index theorem is the following
(see A. Connes, \cite{Connes:1985}).
\begin{theorem}[Connes Index Theorem]\label{Connes:1985}
Let $(A,\pi,F)$ be a $p$-summable  odd Fredholm module and $P$ given by
\begin{equation}
P=\frac{1}{2}(F+I).
\end{equation}
Then, for every invertible element $u\in A,$ the operator $PuP:PH\rightarrow PH$ is Fredholm and its analytical index is given by
\begin{equation}
\textnormal{ind}(PuP)=\frac{(-1)^{2k+1}}{2^{2k+1}}\textnormal{Tr}[a_0[F,a_1]\cdots [F,a_{2k+1}]],
\end{equation} where $p$ is the smallest odd integer larger than $n,$   $a_i=u^{-1}$ for $i$ even, and $a_i=u$ for $i$ odd.
\end{theorem}
In the next section we compute the index of a Toeplitz operator $T_{f}=PM_{f}P.$ In order to use the Connes theorem, we use (the well know fact) that $(C^{\infty}(G),\pi,2P-I)$ is a odd Fredholm module,  where $\pi$ is the representation $\pi(g)=M_{g}$ defined from $C^{\infty}(G)$ into the algebra of bounded operators on $L^2(G).$

\section{The index of Toeplitz operators on compact Lie groups}\label{indexsection}

In this section, we compute the index of Toeplitz operators by using trace formulae for global operators of trace class and the index theorem of Connes mentioned above. The corresponding statement for trace class global operators is the following (for the proof, we refer the reader to Cardona \cite{CardonaIndex}. The proof is based in the arguments developed by Delgado and Ruzhansky \cite{DR,DR1,DR3}).
\begin{theorem}\label{traceondiagonal} Let $A$ be a pseudo-differential operator on $\Psi^{m}(G),$ $m< -\dim(G).$ Then $A$ is trace class on $L^2(G)$ and 
$$ \textnormal{Tr}(A)=\int_{G}\sum_{[\xi]\in\hat{G}}d_{\xi}\textnormal{Tr}[\sigma_{A}(x,\xi)]dx, $$ where $\sigma_A(x,\xi)$ is the matrix-valued symbol of $A.$
\end{theorem}

In order to apply the Connes theorem, we need the following well known  lemma. For completeness we provide a proof.

\begin{lemma}\label{tripleta}
The triple $(C^{\infty}(G),\pi,2P-I)$ where $\pi:C^\infty(G)\rightarrow B(L^2(G)),$ is the representation defined at $g$ by $\pi(g)=M_{g}$  (multiplication operator by $g,$) is an $p$-summable odd Fredholm module for $p>n:=\dim(G)$.
\end{lemma}
\begin{proof}
We only need to prove that $F=2P-I$ is self-adjoint, that $F^2=I$   and the compactness of every commutator $[F,\pi(g)].$ Because $P^2=P$ and $P$ is orthogonal, we deduce that $F$ is self-adjoint and $F^2=4P^2-4P+I=I.$ Now, if $g\in C^{\infty}(G)$ the operators $M_{g}P$ and $PM_g$ same in local coordinates same principal symbol and the order of the commutator $T=[P,M_g]$ is $-1.$ This implies that $|T|\in L^p(L^2(G))$ for $p>n=\dim(G).$ 
\end{proof}
Now, with the machinery presented above, we can give a short argument for proving our main theorem.
\begin{theorem}[Index of Toeplitz operators]\label{maintheoremDC}
Let us consider a smooth  and invertible function $f$ on $G.$  Let  $\Pi:L^2(G)\rightarrow \Pi'$ be the orthogonal projection of a closed subspace  $\Pi'\subset L^2(G)$. If  $T_f=\Pi M_{f} \Pi$ is the Toeplitz operator  with symbol $f,$ then $T_f$ extends to a   Fredhlom operator on $L^2(G)$ and its index is given by
\begin{align}
\textnormal{ind}(T_f)=\int\limits_{G}\sum_{[\xi]\in\hat{G}}d_{\xi}\textnormal{Tr}[\sigma_{I_{\Pi,f}}(x)\xi(e_G)     ]dx,\,\,I_{\Pi,f}:=f^{-1}\underbrace{[{\Pi},f][{\Pi},{{f^{-1}}}]\cdots [{\Pi},f]}_{n+2-times} 
\end{align} if $n$ is odd,  or
\begin{align}
\textnormal{ind}(T_f)=\int\limits_{G}\sum_{[\xi]\in\hat{G}}d_{\xi}\textnormal{Tr}[\sigma_{ I_{\Pi,f} }(x)\xi(e_G)    ]dx,\,\,I_{\Pi,f}:=f^{-1}\underbrace{[{\Pi},f][{\Pi},f^{-1}]\cdots [{\Pi},f]}_{n+1-times},
\end{align}  if $n$ is even, where $\sigma_{I_{\Pi,f}}$ is the Ruzhansky-Turunen operator valued symbol associated to $I_{\Pi,f}.$
\end{theorem}

\begin{proof}
Let us observe that the index of $T$ can be computed as,
\begin{equation}
\textnormal{ind}(T_f)= \textnormal{ind}(\Pi f\Pi)=\frac{(-1)^{2k+1}}{2^{2k+1}}\textnormal{Tr}[a_0[F,a_1]\cdots [F,a_{2k+1}]],
\end{equation} where $p=2k+1> n:=\dim (G) ,$ is the smallest odd integer larger that $n,$ $a_i=f^{-1}$ for $i$ even, and $a_i=f$ for $i$ odd. Here we have used that $(C^{\infty}(G),M_f,2\Pi-I)$ is a $p$-summable odd Fredholm module as well as the Connes index Theorem. Since $F=2P-I,$ we have
\begin{equation}
[F,a_1]\cdots [F,a_{2k+1}]=2^{2k+1}[\Pi,a_1]\cdots [\Pi,a_{2k+1}],
\end{equation} and 
\begin{equation}
\textnormal{ind}(T_f)= \textnormal{ind}(\Pi f\Pi)=-\textnormal{Tr}[a_0[\Pi,a_1]\cdots [\Pi,a_{2k+1}]].
\end{equation} Now we need to compute the trace of the operator $I_{\Pi,f}=a_0[\Pi,a_1]\cdots [\Pi,a_{2k+1}]\in S^{-p}(G),$  $p=2k+1>n,$ using its operator-valued symbol. This can be done with Theorem \ref{traceondiagonal}.  From, Lemma \ref{connection}, we have the following matricial identity
\begin{eqnarray}
\sigma_{I_{\Pi,f}}(x)\xi(e_G)=\sigma_{I_{\Pi,f}}(x,\xi),\,x\in G,\,\,[\xi]\in \widehat{G}.
\end{eqnarray}
From Theorem \ref{traceondiagonal}, we have
\begin{align*}
\textnormal{ind}(T_f) &= \textnormal{ind}(\Pi f\Pi)=-\textnormal{Tr}[I_{\Pi,f}]
=-\int_{G}\sum_{[\xi]\in\hat{G}}d_{\xi}\textnormal{Tr}[\sigma_{I_{\Pi,f}}(x,\xi)]dx.
\end{align*} So, we have proved that
\begin{align}
\textnormal{ind}(T_f)=\int\limits_{G}\sum_{[\xi]\in\hat{G}}d_{\xi}\textnormal{Tr}[\sigma_{I_{\Pi,f}}(x)\xi(e_G)     ]dx,\,\,I_{\Pi,f}:=f^{-1}\underbrace{[{\Pi},f][{\Pi},{{f^{-1}}}]\cdots [{\Pi},f]}_{n+2-times} 
\end{align} if $n$ is odd,  or
\begin{align}
\textnormal{ind}(T_f)=\int\limits_{G}\sum_{[\xi]\in\hat{G}}d_{\xi}\textnormal{Tr}[\sigma_{ I_{\Pi,f} }(x)\xi(e_G)    ]dx,\,\,I_{\Pi,f}:=f^{-1}\underbrace{[{\Pi},f][{\Pi},f^{-1}]\cdots [{\Pi},f]}_{n+1-times},
\end{align}  if $n$ is even. Thus, we finish the proof.

\end{proof}
Finally, we present the following remark on our index theorem and the classical index theorem for Toeplitz operators.
\begin{rem}[Noether-Gohberg-Krein index formula]
Let us consider the operator $D=\frac{1}{ i}\frac{d}{dt}$ on $\mathbb{S}^{1}\equiv[-\pi,\pi),$ and  the odd Fredholm module $(C^{\infty}(\mathbb{S}^1), \pi, F),$ where $F=D|D|^{-1},$ out of $\ker(D)$ and $F=I$ on $\ker(D).$ Then $P=\frac{I+F}{2}$ is the orthogonal projection from $L^2(\mathbb{S}^1)$ into the Hardy space $H^1(\mathbb{S}^1)=\{f\in L^2:\widehat{f}(n)=0,\textnormal{ for }n<0\},$ in fact
\begin{equation}
P(e^{int})=\frac{1}{2}(e^{int}+(-1)^{n}e^{int}), \,\,n\in\mathbb{Z}.
\end{equation}
In this case (see e.g., R. Rodsphon\cite{Rodsphon} Example 3, p. 21, or  N. Higson \cite{Higson}, p. 19)
\begin{equation}
\textnormal{ind}(PM_fP)=-\textnormal{Tr}[f^{-1}[P,f]]=-\frac{1}{2\pi i}\int_{\mathbb{S}^1}f^{-1}df.
\end{equation} So, in our index formula in Theorem \ref{maintheorem} for the abelian compact Lie group $G=\mathbb{S}^1,$ the unitary dual can be identified with $\{e^{inx}:n\in\mathbb{Z}\},$ so if we take $P=\Pi$ as the projection to the subspace generated by $\{e^{inx}:n\geq 0\},$ we obtain $\Pi'=H^1(\mathbb{S}^1)$ (the Hardy space modeled on $L^1(\mathbb{S}^1)$) and we recover the historical result by Noether-Gohberg-Krein \eqref{ToeplitzNGK}.
\end{rem}

\section{The index of Toeplitz operators on closed 3-manifolds }\label{indexsectioncompact}

\subsection{Ruzhansky-Turunen construction}

Let us consider the compact Lie group $\textnormal{SU}(2)\cong \mathbb{S}^3$ consisting of those orthogonal matrices $A$ in $\mathbb{C}^{2\times 2},$ with $\det(A)=1$.   We recall that the unitary dual of $\textnormal{SU}(2)$ (see \cite{Ruz}) can be identified as
\begin{equation}
\widehat{\textnormal{SU} (2)}\equiv \{ [\xi_{l}]:2l\in \mathbb{N}, d_{l}:=\dim \xi_{l}=(2l+1)\}.
\end{equation}

Let us assume that  $\mathbb{M}$  is a simply connected closed 3-manifold, (this means  that every
simple closed curve within the manifold can be deformed continuously to a point). By the Poincar\'e conjecture (c.f. \cite{Poincare})  proved by Perelman (see Perelman \cite{Perelman1,Perelman2,Perelman3,Perelman4} and Morgan\cite{Morgan}), $\mathbb{M}$ is diffeomorphic to $\mathbb{S}^3\simeq \textnormal{SU}(2).$ Every topological 3-manifold admits a differentiable structure and every homeomorphism between
smooth 3-manifolds can be approximated by a diffeomorphism. Thus, classification results
about topological 3-manifolds up to homeomorphism and about smooth 3-manifolds up to diffeomorphism
are equivalent (see Morgan \cite{Morgan}). 

Let $\Phi:\textnormal{SU}(2)\rightarrow \mathbb{M}$ be a diffeomorphism. By following Ruzhansky and Turunen, \cite{Ruz}, pag. 578, $\mathbb{M}$ can be endowed with the natural Lie group structure induced by $\Phi.$ In fact, if $x,y$ are coordinate points in $\mathbb{M}$ we can define
\begin{equation}\label{Mcompact}
    x\cdot y:=\Phi(  \Phi^{-1}(x)\times \Phi^{-1}(x)   ),
\end{equation}where $\times$ denotes the product of matrices on  $\textnormal{SU}(2). $ We have a Frechet isorphism $C^\infty(\textnormal{SU}(2))\simeq C^\infty(\mathbb{M}) $ defined by 
\begin{equation}
\Phi_*:C^\infty(\textnormal{SU}(2))\rightarrow C^\infty(\mathbb{M}),\,\,\Phi_*(f):=f\circ \Phi^{-1};\,\,
\Phi^*:C^\infty(\mathbb{M})\rightarrow C^\infty(\textnormal{SU}(2)),\,\,\Phi_*(g):=g\circ \Phi. 
\end{equation} Since $L^2(M)=\Phi_{*}(L^2(\textnormal{SU}(2))),$ and the Lie group structure on $\mathbb{M}$ provides a Peter-Weyl theorem on  $\mathbb{M}, $  we have $\widehat{\textnormal{SU}(2)}=\frac{1}{2}\mathbb{N}_0\simeq \widehat{\mathbb{M}}$ in the sense that 
\begin{eqnarray}
\Phi_*:\widehat{\textnormal{SU}(2)  }\rightarrow \widehat{M},\,\,\Phi_{*}[\xi_\ell]=[\Phi_*\xi_{\ell}],\,\,\Phi_*\xi_\ell:=\xi_\ell\circ \Phi\equiv [\xi_{\ell,ij}\circ \Phi]_{i,j=-\ell}^{\ell},\,\ell\in\frac{1}{2}\mathbb{N}_0,
\end{eqnarray} is a well defined isomorphism. We have used that  $\Phi_*:C^\infty(\textnormal{SU}(2))\rightarrow C^\infty(\mathbb{M})$ extends to a linear unitary bijection from $L^2(\textnormal{SU}(2))$ into $L^2(\mathbb{M}) ,$ via
\begin{equation}
\langle g,h\rangle_{L^2(\mathbb{M})}=\langle g\circ \Phi,h\circ \Phi\rangle_{L^2(\textnormal{SU}(2))}=\langle  \Phi_*(g),\Phi_*(h) \rangle_{L^2(\textnormal{SU}(2))}.
\end{equation}
As it was pointed out in \cite{Ruz}, this immediately implies that the whole construction of matrix-valued symbols on $\mathbb{M}$ is equivalent to that on $\textnormal{SU}(2).$
Because we can endowed to $\mathbb{M}$ with a Lie group structure, Theorem \ref{maintheorem} can be used to analise the index of Toeplitz operators on $L^2(\mathbb{M}).$

\subsection{Toeplitz operators}
Theorem \eqref{maintheoremDC} implies the following result.
\begin{corollary}\label{maintheoremDC2}
Let us consider a smooth  function and invertible function $\omega$ on $\mathbb{M}.$  Let  $\Pi:L^2(\mathbb{M})\rightarrow \Pi'$ be the orthogonal projection of a closed subspace  $\Pi'\subset L^2(\mathbb{M})$. If  $T_\omega=\Pi M_{\omega} \Pi$ is the Toeplitz operator  with symbol $\omega,$ then $T_\omega$ extends to a  extends to a Fredhlom operator on $L^2(\mathbb{M})$ and its index is given by
\begin{align}
\textnormal{ind}(T_\omega)=\int\limits_{\mathbb{M}}\sum_{\ell \in \frac{1}{2}\mathbb{N}_0}(2\ell+1)\textnormal{Tr}[\sigma_{ \omega^{-1} [\Pi,\omega][\Pi,\omega^{-1}][\Pi,\omega][\Pi,\omega^{-1}][\Pi,\omega]} (x)   \xi_\ell(e_\mathbb{M}))     ]dx, 
\end{align} where $e_\mathbb{M}=\Phi(e_\textnormal{SU(2)}).$
\end{corollary}
\begin{proof}
With the product defined in \eqref{Mcompact}, $\mathbb{M}  $ has the Lie group structure diffeomorphic to that on $\textnormal{SU}(2).$ The previous Ruzhansky-Turunen construction gives $\widehat{\mathbb{M}}\simeq \frac{1}{2}\mathbb{N}_0,$ and every unitary and strongly continuous unitary representation on $\widehat{\mathbb{M}}$ has the form $\Phi_*\xi_\ell:\mathbb{M}\rightarrow \textnormal{U}(\mathbb{C}^{d_{\xi_\ell}}),$ $\ell\in\frac{1}{2}\mathbb{N}_0,$  $\dim(\Phi_*\xi_\ell)=d_{\xi_\ell}=2\ell+1.$ So, the proof now follows from Theorem \ref{maintheoremDC2}.
\end{proof}
Now, we study the previous formula in local coordinates.
\begin{rem}
By using the diffeomorphism 
$\varrho:\mathbb{M}\simeq \textnormal{SU}(2)\rightarrow \mathbb{S}^3,$ defined by
\begin{equation}\varrho(z)=x:=(x_1,x_2,x_3,x_4),\,\,\,\textnormal{for}\,\,\,\,
z=\begin{bmatrix}
    x_1+ix_2       & x_3+ix_4  \\
    -x_3+ix_4       & x_1-ix_2 
\end{bmatrix}, 
\end{equation} we have
\begin{align*}
&\textnormal{ind}(T_\omega)\\
&=\int\limits_{\mathbb{M}}\sum_{\ell \in \frac{1}{2}\mathbb{N}_0}(2\ell+1)\textnormal{Tr}[\sigma_{ \omega^{-1} [\Pi,\omega][\Pi,\omega^{-1}][\Pi,\omega][\Pi,\omega^{-1}][\Pi,\omega]} (z)   \xi_\ell(e_\mathbb{M}))    ]dz, \\
&=\int\limits_{\mathbb{S}^3}\sum_{\ell \in \frac{1}{2}\mathbb{N}_0}(2\ell+1)\textnormal{Tr}[\sigma_{ \omega^{-1} [\Pi,\omega][\Pi,\omega^{-1}][\Pi,\omega][\Pi,\omega^{-1}][\Pi,\omega]} (x)   \xi_\ell(e_\mathbb{M}))      ]d\tau(x), \\
\end{align*}where  $$\sigma_{ \omega^{-1} [\Pi,\omega][\Pi,\omega^{-1}][\Pi,\omega][\Pi,\omega^{-1}][\Pi,\omega]}(\varrho^{-1}(x))=:\sigma_{ \omega^{-1} [\Pi,\omega][\Pi,\omega^{-1}][\Pi,\omega][\Pi,\omega^{-1}][\Pi,\omega]}(x),\,\,z=\varrho^{-1}(x).$$ If we use the parametrization of $\mathbb{S}^3$ defined by $x_1:=\cos(\frac{t}{2}),$ $x_2:=\nu,$ $x_3:=(\sin^2(\frac{t}{2})-\nu^2)^{\frac{1}{2}}\cos(s),$ $x_4:=(\sin^2(\frac{t}{2})-\nu^2)^{\frac{1}{2}}\sin(s),$ where
$$(t,\nu,s)\in D:=\{(t,\nu,s)\in\mathbb{R}^3:|\nu|\leq \sin(\frac{t}{2}),\,0\leq t,s\leq 2\pi\},$$
then $d\tau(x)=\sin(\frac{t}{2})d\nu dt ds,$ and 
\begin{align*}
&\textnormal{ind}(T_\omega)=\int\limits_{0}^{2\pi}\int\limits_{0}^{2\pi}\int\limits_{  -\sin(t/2) }^{ \sin(t/2)   }\sum_{\ell\in \frac{1}{2}\mathbb{N}_0}   (2\ell +1)\textnormal{Tr}[\sigma_{ \omega^{-1} [\Pi,\omega][\Pi,\omega^{-1}][\Pi,\omega][\Pi,\omega^{-1}][\Pi,\omega]}(\nu,t,s)].
\end{align*} Thus, we have obtained an explicit index formula for Toeplitz operators on compact 3-manifolds (with trivial fundamental group).
\end{rem}

\textbf{Acknowledgement.} I would like to thank Alexander Cardona  for suggesting the problem of computing the index of Toeplitz operators.

\end{document}